\documentclass[a4paper,12pt]{amsart}
\usepackage[latin1]{inputenc}
\usepackage{amsmath,amsthm,amsfonts,amscd,amssymb,eucal,latexsym,mathrsfs}
\usepackage[all]{xy}
\usepackage{latexsym, amssymb, amscd, mathrsfs, graphics, fullpage,graphicx}

\theoremstyle{definition}
\newtheorem{theorem}{Theorem}
\newtheorem{corollary}[theorem]{Corollary}
\newtheorem{lemma}[theorem]{Lemma}
\newtheorem{proposition}[theorem]{Proposition}
\newtheorem{remark}[theorem]{Remark}

\newtheorem{definition}[theorem]{Definition}

 \DeclareMathOperator{\Aut}{Aut} 
  \DeclareMathOperator{\Inn}{Inn}

\DeclareMathOperator{\supp}{ supp}

\newcommand{\reg}{r}

\newcommand{\HI}{{\mathrm{\bf H}}}\newcommand{\CI}{{\mathrm{\bf C}}}\newcommand{\RI}{{\mathrm{\bf R}}}\newcommand{\QI}{{\mathrm{\bf Q}}}
\newcommand{\ZI}{{\mathrm{\bf Z}}}\newcommand{\NI}{{\mathrm{\bf N}}}

\newcommand{\G}{\Gamma}\renewcommand{\S}{S}

\newcommand{\SL}{\mathrm {SL}}\newcommand{\PSL}{\mathrm {PSL}}

\newcommand{\inj}{\hookrightarrow}

\newcommand{\sq}{${7\over 4}$}

\date{February 2, 2009}

\title{The 4-string braid group $B_4$ has property RD\\ and exponential mesoscopic  rank}

\author{Sylvain Barr\'e}
\address{\hskip-\parindent
Sylvain Barr\'e, Universit\'e de Bretagne Sud, BP 573, Centre Yves Coppens, Campus de Tohannic, 56017 Vannes, France}
\email{sylvain.barre@univ.ubs.fr}
\author{Mika\"el Pichot}
\address{\hskip-\parindent
Mika\"el Pichot, Department of Mathematical Sciences,
University of Tokyo,
3-8-1 Komaba, Tokyo, 153-8914
JAPAN}
\email{pichot@ms.u-tokyo.ac.jp}

\begin{document}

\begin{abstract}
We prove that 
the braid group $B_4$ on 4 strings, as well as its central quotient $B_4/\langle z\rangle$, have the property RD of Haagerup--Jolissaint. 
It follows that the automorphism group $\Aut(F_2)$ of the free group $F_2$ on 2 generators has property RD.

We also prove that the braid group $B_4$ is a group of intermediate rank (of dimension 3). Namely, we show that  
 both $B_4$  and its central quotient have exponential mesoscopic rank, i.e., that they contain exponentially many large flat balls which are not included in flats. 
\end{abstract}

\maketitle

\section{Introduction}

Let $n\geq 2$ be an integer. The braid group $B_n$ on $n$ strings is a finitely presented group generated by $n-1$ elementary braids $\sigma_1,\ldots,\sigma_{n-1}$ subject to the following relations: 
\begin{itemize}
\item $\sigma_i\sigma_{i+1}\sigma_i=\sigma_{i+1}\sigma_i\sigma_{i+1}$ for all $1\leq i\leq  n-2$; 
\item $\sigma_i\sigma_j=\sigma_j\sigma_i$ for all $1\leq i,j\leq n-1$ such that  $|i-j|\geq 2$.
\end{itemize}
This is the classical Artin presentation of $B_n$ (see e.g. Chapter 10 in \cite{BZ}).  

The group $B_3$ is  closely related to the modular group $\PSL_2(\ZI)$. The above presentation shows that the braid $z=(\sigma_1\sigma_2)^3$ is central in $B_3$  and that $B_3/\langle z\rangle$ is generated by the class $u$ of $\sigma_1\sigma_2\sigma_1$ and $v$ of $\sigma_1\sigma_2$, where $u^2=v^3=z$. Thus $B_3/\langle z\rangle=\langle u,v\mid u^2=v^3=1\rangle=\PSL_2(\ZI)$. In fact the group $B_3$ admits a proper isometric action with compact quotient on a metric product $T_3\times \RI$, where $T_3$ is a trivalent tree,  which is the Bass-Serre tree of $\PSL_2(\ZI)$.

We are interested  here in the 4-string braid groups $B_4$. It was proved by Brady in \cite{Brady} that $B_4$ admits a free isometric action with compact quotient on a CAT(0) simplicial complex $Y$ of dimension 3. The 3-dimensional cells of $Y$ are Euclidean tetrahedra whose faces are right-angle triangles and the quotient space $Y/B_4$ contains 16 tetrahedra, identified together along a single vertex.
It is still true that $Y$ splits as a product $Y=X\times \RI$, where  $X$ is now of dimension 2. The complex $X$ can be obtained from a non positively curved complex of groups whose fundamental group is the quotient of $B_4$ by its center (see \cite{crisp}). 

The existence of a CAT(0) structure on $B_n$ is  an open problem for $n\geq 6$. Recall that on $B_4$, the 3-dimensional CAT(0) structure which are minimal (e.g., those whose links are isomorphic to that of $Y$) can be classified, by geometric rigidity results  due to  Crisp and Paoluzzi \cite{crisp}. 
On the other hand, Charney \cite{Charney-b4} proved that the Deligne complex \cite{Deligne} of $B_4$ is also  a CAT(0) space of dimension 3, with respect to the Moussong metric (we remind that the Deligne action of $B_4$ on this complex is not proper). 

\subsection{Property RD} Let now $G$ be an arbitrary countable group. A length on $G$ is a map $|\cdot| : G\to \RI_+$ such that $|e|=0$, $|s|=|s^{-1}|$ and $|st|\leq |s|+|t|$ for $s,t\in G$ and $e$ the identity element. We recall that  $G$ is said to have \emph{property RD} (\cite{Jol-def}) with respect to a length $|\cdot|$ if there is a polynomial $P$ such that for any $r\in \RI_+$
and $f,g\in \CI G$  with $\supp(f)\subset B_r$ one has 
\[
\|f*g\|_2\leq P(r)\|f\|_2\|g\|_2
\]  
where $B_r=\{x\in G,~|x|\leq r\}$ is the ball of radius $r$ in $G$, $\supp(f)$ is the set of $x\in G$ with $f(x)\neq 0$,  and $\CI G$ is the complex group algebra of $G$. For an introduction to property RD we refer to Chapter 8 in \cite{Val-bc}. The above convolution inequality, usually referred to as the Haagerup inequality (after  Haagerup \cite{Haa}), allows to control the  operator norm of $f$ acting by convolution on $\ell^2(G)$ in terms of its $\ell^2$ norm. Hence, some important consequences of property RD are of a spectral nature. 

When $G$ is finitely generated we have the word length $|\cdot|_S$  associated to any finite generating set $S$. Then property RD with respect to $|\cdot |_S$ is independent of $S$  so we simply speak  of property RD for $G$ is that case.

Our first main result is the following theorem.

\begin{theorem}\label{th1}
The braid group $B_4$ on 4 strings, as well as its central quotient $B_4/\langle z\rangle$, have property RD.
\end{theorem}

This gives a partial answer to a question in \cite{questions}, Section 8.  The fact that $B_3$ has property RD was shown very early on  by Jolissaint in \cite{Jol-def}, and the other cases remained open since then;  in \cite{questions}, the question of property RD is raised more generally for all braid groups $B_n$. 

\emph{Update.} The problem of showing property RD for $B_n$ has been solved recently by Behrstock and Minsky (see \cite{bermin}).
More generally, they established property RD for all mapping class groups  (Recall that the braid group $B_n$ can be identified to the mapping class group of the $n$-punctured disk.)

\medskip

The proof of Theorem \ref{th1} is divided into two steps. 
The first step relies on our previous results from \cite{rd}:

\begin{theorem}[{\cite[Theorem 5]{rd}}]\label{th3}
Let $G$ be a group acting properly on  a CAT(0) simplicial complex $\Delta$ of dimension 2 without boundary and whose faces are equilateral triangles of the Euclidean plane. Then $G$ has property RD with respect to the length induced from the 1-skeleton of $\Delta$. 
\end{theorem}

We apply Theorem \ref{th3} to the quotient   $B_4/\langle z\rangle$.  By results of \cite{Brady,crisp}, this group  acts on a simplicial complex $X$ with the required properties.

The second step uses automaticity of $B_4$, and more precisely, the Thurston normal forms for braids in $B_4$, which allows  to go back to $B_4$ from its central quotient. 

Details of the proof are  in Section \ref{s2}, after a brief survey on property RD in Section \ref{s1'}. It would be interesting to implement the above approach  of solving first the case of central quotients  for higher braid groups. 

\medskip

As a corollary of Theorem \ref{th1}, we obtain the following result (compare \cite{questions}, Section 8, where the question of property RD is raised  in general for all $\Aut(F_n)$, $n \geq 2$):

\begin{corollary}\label{c3}
The automorphism group $\Aut(F_2)$ of the free group on 2 generators has property RD.
\end{corollary}

Indeed, $\Aut(F_2)$ is isomorphic to $\Aut(B_4)$, itself containing  $\Inn(B_4)$ as a subgroup of index 2 (see \cite{dyer,Kram}). Thus property RD for $\Aut(F_2)$ follows from the corresponding result  for $\Inn(B_4)$, which is isomorphic to the central quotient of $B_4$. Note that in \cite{prw}, a faithful action of $\Aut(F_2)=\Aut(B_4)$ on the complex $X$ is constructed.

\subsection{The braid group $B_4$ as a group of intermediate rank} Groups and simplicial complexes appearing in Theorem \ref{th3} provide us with a large pool of objects satisfying \emph{intermediate rank} properties. See \cite{rd} for definitions and concrete examples. 
We discuss here the intermediate rank properties of $B_4$ and its central quotient (denoted  $G$ below). 

We introduced in  \cite{rd}  a notion of \emph{mesoscopic rank} for a CAT(0) space $X$, which reflects the presence in $X$ of maximal flats portions (where maximal refers to the dimension, hence the rank terminology) which are (much) larger than  ``germs of flats" in $X$ (say,  flats of tangent cones) but \emph{are not actually contained in genuine flats of $X$} (i.e. copies of the Euclidean space $\RI^n$ inside $X$). We recall the precise definitions of mesoscopic rank and exponential mesoscopic rank in Section \ref{s3}.
Following  \cite{rd} we say that a group $G$ is of (exponential) mesoscopic rank when there is a proper action of $G$ with compact quotient on some CAT(0) space  which is of (exponential) mesoscopic rank at some point. 

Our second main result is as follows.

\begin{theorem}\label{meso}
The braid group $B_4$ on 4 strings is of exponential mesoscopic rank.
\end{theorem}

For the proof, we first establish  that the quotient $G$ of $B_4$ by its center is of exponential mesoscopic rank, and then extend the result to $B_4$.
Note that $B_3$ is  an example of a group acting freely and cocompactly on a simplicial complex as in Theorem \ref{th3} (see \cite{BradyCam}) but it is not of mesoscopic rank, and more precisely for any action with compact quotient on a 2-dimensional CAT(0) space $X$, the space $X$ cannot be of mesoscopic rank.

In course of proving Theorem \ref{meso} we will see that the central quotient $G$ of $B_4$ is,  at the local level,  closely related  to affine Bruhat-Tits buildings of type $\tilde A_2$  (what actually creates some complications in the proof of Theorem \ref{meso}, since the latter are not of mesoscopic rank by \cite{rd}). We will prove however that these connections cannot be extended beyond the local level (and specifically beyond the sphere of radius 1,  see  the last section of the paper). Related to this, we also show that being of exponential mesoscopic rank cannot serve as an obstruction to being embeddable in an affine building, and in particular, in spaces which are not of mesoscopic rank. 

\bigskip

\emph{Acknowledgments.} We thank Jason Behrstock for communicating us his recent preprint \cite{bermin} with Yair Minsky,  as well as for the reference \cite{prw}. The second author thanks JSPS for support.

\section{Property of rapid decay}\label{s1'}

In \cite{Haa}  Haagerup proved that,  for any finitely supported functions $f,g: F_n \to \RI$ defined on the free group $F_n$ on $n$ generators, the convolution product satisfies the inequality
\[
\|f*g\|_2\leq (r+1)\|f\|_2\|g\|_2
\] 
where $r$ is the radius of the support of $f$, with respect to the usual word-length metric of $F_n$. In other words $f$, viewed as a convolution operator from $\ell^2(F_n)$ to itself, is bounded with operator norm at most $(r+1)\|f\|_2$. 

Groups satisfying the above inequality with some polynomial $P(r)$ instead of $r+1$ are said to have property RD  (the precise definition of which we recalled in the introduction), see \cite{Jol-def}, where Jolissaint showed that  (with respect to the word length):
\begin{itemize}
\item a finitely generated amenable group has property RD if and only if it is of polynomial growth;
\item uniform lattices in a rank 1 Lie group have property RD. 
\end{itemize}

The latter has been extended to all hyperbolic groups in the sense of Gromov by de la Harpe \cite{Harpe-rd}, and subsequently to groups which are hyperbolic relatively to polynomial growth subgroups by Chatterji and Ruane \cite{Chat-Ruane},   thereby establishing property RD for all lattices (uniform or not) in rank 1 Lie groups.

The situation is different for groups of rank $\geq 2$.  Non uniform lattices in a higher rank Lie group, typically  $\SL_3(\ZI)$, are  prominent examples of groups without property RD (cf. \cite{Jol-def}). Valette conjectured that all uniform lattices in higher rank Lie groups have property RD. This is known to hold for uniform lattices in $\SL_3(\QI_p)$ (and other groups acting on triangle buildings), by a well-known theorem of Ramagge--Robertson--Steger \cite{RRS} (see also \cite{Laf-rd}) which was the first occurrence of property RD  in higher rank situations. Their results were  extended by Lafforgue \cite{Laf-rd} to cover all uniform lattices in $\SL_3(\RI)$ and $\SL_3(\CI)$. Chatterji \cite{Chat-quat}  showed then that lattices in $\SL_3(\HI)$ and $E_{6(-26)}$  behave similarly. 

We refer the interested reader to \cite{Val-bc,questions} for more information. A well-known application of property RD concerns the Baum--Connes conjecture without coefficient: 
by a theorem of Lafforgue \cite{Laf-bc},  groups which satisfy property RD together with some non positive curvature assumption (called strong bolicity) also satisfies the Baum--Connes without coefficient. For groups with property T, including most hyperbolic groups or cocompact lattices $\SL_3(\RI)$ (for instance), this is the only known approach to the Baum--Connes conjecture. (The Baum--Connes conjecture is open for $\SL_3(\ZI)$.)

In \cite{rd} we studied   ``rank interpolation"  for countable groups, that is, interpolation of the rank in between the usual $\mathrm{rk}= 1,2,\ldots$ integer values.  
The main applications presented in \cite{rd}  are $C^*$-algebraic in nature and in particular, we established property RD for many groups of intermediate rank. This provided new examples where Lafforgue's approach to the Baum--Connes could be applied (in fact for many of these groups---e.g. for groups of rank \sq---this is also the only approach that is presently known to work, and the Baum-Connes conjecture with coefficients is open). See also \cite{notewise} and \cite{bs} for other results on intermediate rank and property RD. 
The accent in \cite{rd} is on interpolating the rank between 1 and 2, which includes a large class of groups of interest. In the present paper we will see that $B_4$ is also a group of  intermediate rank, which interpolate the  rank between 2 and 3.

\section{Proof of Theorem \ref{th1}}\label{s2}

The group $B_4$ admits the following presentation: 
\[
B_4=\langle a,b,c\mid aba=bab, bcb=cbc, ac=ca\rangle.
\]
The pure braid group $P_4$ is the kernel of the surjective homomorphism to the symmetric group on 4 letters,
\[
B_4\to S_4,
\]
mapping a braid to the corresponding permutation of its endpoints. 
It is well-known that the center of both $B_4$ and $P_4$ is the cyclic group generated by the element $z=(bac)^4$, which consists in a full-twist braiding of the 4 strings (see \cite[Section 10.B]{BZ} for instance; this is known to hold for more general Artin groups \cite{BS,Deligne}). 
In other words $B_4$ is a central extension of the group 
\[
G=B_4/\langle z\rangle
\]
by the groups of integers $\ZI=\langle z\rangle$, which gives an exact sequence
\[
1\longrightarrow \ZI\longrightarrow B_4\longrightarrow G\longrightarrow 1,
\]
and in the same way,
\[
1\longrightarrow \ZI\longrightarrow P_4\longrightarrow H\longrightarrow 1,
\]
where $H=P_4/\ZI$ is a finite index subgroup of $G$. The torsion in $G$ corresponds to the the conjugacy classes of the elements $x=bac$ and $y=bac^2$ and their powers, where we have $x^4=y^3=z$ (see \cite[p. 139]{crisp} for a geometric proof of this fact; recall that $G_4$ itself is torsion free). It follows that $H$ is torsion free. 

We will need some results of Brady \cite{Brady} and their extensions in Crisp--Paoluzzi \cite[Section 3]{crisp}. Let $Y$ be classifying space of $B_4$  constructed in \cite{Brady}. As recalled in the introduction, $Y$ is a CAT(0) simplicial complex of dimension 3 whose 3-dimensional faces are Euclidean tetrahedra. The authors of \cite{crisp} consider the projection in $Y$ along the $z$-axis and obtain a 2-dimensional complex $X$ (called the Brady complex there) together with an action of $G$  (called the standard action, in view of \cite[Theorem 1]{crisp}) which commutes to the action of $B_4$ on $\Sigma$ under taking projection. The complex $Y$ splits metrically as a product: 
\[
Y=X\times \RI
\] 
and $X$ is endowed with an action of $G$ (in Section \ref{s4} we will give more details on these constructions). 

As a CAT(0) space, $X$ is a triangle polyhedron, i.e. its faces are equilateral triangles of the Euclidean plane (\cite[p. 140]{crisp}), and the action of $G$ on $X$ is proper with compact quotient. Thus $H$ acts freely with compact quotient on $X$, so $H$ appears as  the fundamental group of the complex 
\[
V=X/H
\]
(it can be shown that $V$ has 6 vertices and 32 faces). It follows then from Theorem \ref{th3} that  $H$ has property RD with respect to some (hence any) finite generating set. As $H$ is a finite index normal subgroup of $G$ this implies, by Proposition 2.1.4 in  \cite{Jol-def},   that $G$ itself has property RD.

Further results of Jolissaint \cite{Jol-def} (in particular Proposition 2.1.9 of that paper, see also  Chatterji--Pittet--Saloff-Coste \cite[Proposition 7.2]{CPSC}) show that property RD is stable under certain types of central extensions. We will prove that these results can by applied to the present situation and this will conclude the proof of Theorem \ref{th1}.

Consider the section 
\[
\kappa : G\to B_4
\]
of the quotient map $\pi : B_4\to G$, which identifies $G$ as the subset of braids in $B_4$ whose central part is trivial. Being a central extension of $G$, we can decompose $B_4$ as a  product
\[
B_4=\ZI\times_c G
\] 
where the value at a point $(g,h)\in G\times G$  of the cocycle  
\[
c : G\times G \to \ZI
\]
defining the extension is the exponent of $z\in G$ in the central element 
\[
\kappa(g)\kappa(h)\kappa(gh)^{-1}
\] 
of $G$.

Our goal is to find a symmetric finite generating set of $G$ such that, for the corresponding Cayley graph $Y_G$ of $G$, we have
\[
|c(g,h)|\leq n
\]
for every elements $g,h\in G$ at distance at most $n$ from the identity in $Y_G$.
  That this implies property RD for $B_4$  follows from   \cite[Proposition 7.2]{CPSC}.

Let us fix some notations regarding the Thurston normal form for elements of $B_4$ (see \cite[Chapter 9]{thurston} and \cite{Charney-biaut}). 
In what follows we write $\Delta=(bac)^2$ for the half twist of the four strings. 

The braid group $B_4$ can be generated by a set $S$ of 23 elements, which are in bijective correspondence with the non trivial elements  of the symmetric quotient $S_4$. The half-twist $\Delta$ belongs to $S$. Furthermore in this presentation, the monoid $B_4^+$ of positive braids is the submonoid of $G$ generated by $S$, and  every elements $s\in B_4^+$ can be written in a canonical way 
\[
s=s_1\ldots s_n,
\]
called the greedy form of $s$, where $s_i\in S$  (see \cite{Garside,thurston}, for instance one can consider the right greedy form where the element $\Delta$ appears only on the right side of the expression $s_1\ldots s_n$). This decomposition can be extended to $B_4$:  by  \cite{thurston}, every $x\in B_4$ can be written as $x=s^{-1}t$ with $s,t\in B_4^+$, in a unique way (after obvious cancellation in case both  $s$ and $t$ start with the same element $r\in B_4^+$). Thus any elements $x\in B_4$ can be written in a canonical form
\[
x=s_n^{-1}\ldots s_1^{-1}t_1\ldots t_m,
\]
where $s_i,t_j\in S$. The latter decomposition is called the  Thurston normal form (or the Garside normal form) of $x$. Following \cite{charney-n}, we let 
\[
|x|=n+m,
\]
where $n$ and $m$ are given by the  normal form. The language associated to this normal form turns out to give  a geodesic biautomatic structure on $B_4$ (see \cite{thurston,charney-n}), and if we denote by $Y$ the cayley graph of $B_4$ with respect to $S\cup S^{-1}$, then $|x|$ is the length of a simplicial geodesic in $Y$ from $e$ to $x\in B_4$. In particular for $x,y\in B_4$ we have
\[
|xy|\leq |x|+|y|
\]
(see  \cite[Lemma 3.4]{charney-n}).

Let $Y_G$ be the Cayley graph of $G$ with respect to the generating set $\pi(S\cup S^{-1})$. It is easily seen that 
\[
|\kappa(gh)|\leq |\kappa(g)|+|\kappa(h)|
\]
since $\kappa(gh)$ is obtained from the product $\kappa(g)\kappa(h)$ by cancellation of the central factor. In particular 
\[
\kappa(\mathrm{Ball}_n(Y_G))\subset \mathrm{Ball}_n(Y),
\]
where $\mathrm{Ball}_n(\cdot)$ is the ball of radius $n$ in the corresponding Cayley graph. On the other hand, since $\Delta\in S$ and $z=\Delta^2$, the absolute value of the exponent of $z$ in the central part of an $x\in B_4$ is at most $|x|/2$ by construction of the normal form of $x$.

Let $g,h\in G$ at distance at most $n$ from the identity in $Y_G$. By definition, the value of $c(g,h)$ is the exponent of $z$ in the central part of  $\kappa(g)\kappa(h)$. Thus
\[
|c(g,h)|\leq |\kappa(g)\kappa(h)|/2\leq (|\kappa(g)| + | \kappa(h)|)/2\leq n.
\]
This concludes the proof of Theorem \ref{th1}.

\section{Some classical applications of property RD}\label{s4}

We present below two classical applications of property RD. The first one concerns the Baum-Connes conjecture and the second one is about random walks, which gives use some useful information on random walks on $B_4$. For further consequences of property RD we refer to Valette's book \cite{Val-bc} and to the references there. 

\subsection{Braid groups  and the Baum--Connes conjecture}

As is well-known, the Baum--Connes conjecture \emph{with coefficients}  holds for the $n$-string pure braid group $P_n$, as well as for its finite extension $B_n$  (see \cite{oyono,schick}).  On the other hand,  in the case $n=4$, we have property RD  and thus the  Banach $KK$-theory techniques of Lafforgue \cite{Laf-bc} applies as well. Hence:

\begin{corollary}  The groups  $B_4$, $P_4$ and their respective central quotients,  $G$ and $H$, satisfy the Baum--Connes conjecture without coefficient. 
\end{corollary}

The Baum--Connes conjecture (even without coefficients) has a number of applications. See \cite{Val-bc} for more details.
The problem of showing the Baum-Connes conjecture with coefficients for  groups  acting freely isometrically with compact quotient on a CAT(0) space satisfying the assumption of Theorem \ref{th3} is open. 

As far as we know, the Baum--Connes conjecture for the central quotients of $B_n$ and $P_n$ is open for $n\geq 5$.

\subsection{$\ell^2$ spectral radius of random walks on $B_4$} Another application of property RD concerns random walks on groups, see  Grigorchuk and Nagnibeda  \cite{Grigo} and the end of Section 2.2 in \cite{rd} for more details and references. 

If $G$ is a countable group endowed with a length, one considers the \emph{operator growth function of $G$},  
\[
F_\reg(z)=\sum_{n} a_n z^n
\]
where the coefficients $a_n$ are bounded operators on $\ell^2(G)$ defined by  
\[
a_n =\sum_{|x|=n} u_x
\]
with $u_x$, $x\in G$,  the canonical family of unitary operators corresponding to $G$ in $C^*_\reg(G)$ under the regular representation.
The radius of convergence $\rho_\reg$ of $F_\reg$ defined by 
\[
{1\over {\rho_\reg}}=\limsup_{n\to \infty} \|a_n\|_\reg^{1/n}
\] 
is no lower than the radius of convergence of the usual  growth series of the group $G$ with respect to $\ell$. Conjecture 2 in \cite{Grigo} states that $G$ is amenable if and only if $\rho=\rho_\reg$. For groups with property RD (in fact ``radial subexponential" property RD is sufficient, see   \cite[Proposition 23]{rd} and  references) we have $\rho_\reg=\sqrt{\rho}$ and thus the above Conjecture 2 holds. One can also deduce the $\ell^2$ spectral radius property for every element in the group algebra of $G$ provided $G$ has (subexponential) property RD, i.e., the fact that  the spectral radius of every element $a\in \CI G$ acting by convolution on $\ell^2(G)$  is equal to
\[
\lim_{n\to\infty} \|a^{*n}\|_2^{1/n}
\]
(which also has some important applications, again see the references in \cite{rd}). Thus we obtain:

\begin{corollary}\label{cor}  The groups $B_4$, $P_4$ and  their respective central quotients,  $G$ and $H$, satisfy the $\ell^2$ spectral radius property. Furthermore for these four groups the reduced spectral radius $\rho_\reg$ and the radius of convergence $\rho$ of the usual growth series  are related as follows:
\[ 
\rho_\reg=\sqrt{\rho}<1,
\]
and thus these groups satisfy Conjecture 2 in \cite{Grigo}.
\end{corollary}

\section{Mesoscopic rank}\label{s3}

Let $X$ be a piecewise Euclidean CAT(0) simplicial complex of dimension $n\geq 2$, without boundary, and let $A$ be a point of $X$ (see \cite{BH} for a general reference on CAT(0) spaces).
 We call \emph{mesoscopic rank profile} of $X$ at $A$ the function 
\[
\varphi_A : \RI_+\to \NI
\] 
which associate to an $r\in \RI_+$ the number of distinct flat balls of radius $r$ in $X$ which are centered at $A$, and which are not included in a flat of $X$.
By  a flat in $X$ (resp. flat subset of $X$) we mean an isometric embedding of the Euclidean space $\RI^n$ in $X$ (resp. of a subset of an Euclidean space $\RI^n$ with the induced metric). 

We then have  the following.

\begin{proposition}[see \cite{rd}]\label{p7}
Let $X$ be a piecewise Euclidean CAT(0) simplicial complex without boundary and let $A$ be a point of $X$. Then,
\begin{enumerate}
\item if $X$ is hyperbolic,  $\varphi_A$ is compactly supported;
\item if $X$ is an affine Bruhat-Tits building, $\varphi_A$ vanishes identically.
\end{enumerate} 
\end{proposition}

We refer to \cite[Section 6]{rd}, where this theorem is stated for triangle polyhedra but the proof extends to the above general situation (in the first case there is no flat at all, while in the second, we have in fact that every flat ball is included in uncountably many flats).

According to Proposition \ref{p7},  the mesoscopic rank profile trivializes when the rank takes the usual $\mathrm{rk}=1,2,3,\ldots$ integer values. The following property detects spaces of intermediate rank where, more precisely, intermediate rank occurs  (exponentially) in between the local and asymptotic scale in  $X$:

\begin{definition}
The space $X$ is said to have \emph{exponential mesoscopic rank} at $A$ if the function $\varphi_A$ converges exponentially fast to infinity at infinity.   
\end{definition}

Mere \emph{mesoscopic rank} refers to the fact that the support of $\varphi_A$ contains a neighbourhood of infinity. Thus for spaces of mesoscopic rank at a point $A$, on can continuously rescale the radius of balls of center $A$ from some constant $C$ up to $\infty$, in such a way that all the balls in this family are flat but not  included in flats. When the mesoscopic rank is exponential, the number of possible choices for these balls varies exponentially with the radius.

\begin{definition}
A group $G$ is said to be of \emph{exponential mesoscopic rank} if it admits a proper isometric action with compact quotient on a CAT(0) space which is of exponential mesoscopic rank at least at one point (and thus at infinitely many points).  
\end{definition}

The following groups are known to be of exponential mesoscopic rank:

\begin{itemize}
\item[(a)] The group denoted $\G_{\bowtie}$ in \cite{rd}, and called group of frieze there (see \cite[Section 6.1]{rd}); 
\item[(b)] The group of rank \sq\ which is the fundamental group of the  complex denoted $V_0^1$ in \cite{rd} (see \cite[Section 6.1]{rd});
\item[(c)] D. Wise's non Hopfian group (see \cite{notewise,bs}).
\end{itemize}

In the present paper we add further groups to this list, namely $B_4$ and its central quotient (as well as the group $G_0$ of Section \ref{more}).

\begin{remark}
Most  of the groups of rank \sq\  (see \cite[Section 4]{rd}) might be of exponential mesoscopic rank. Besides the one of Item (b) above, one can get more examples from the classification of transitive orientable  groups of rank \sq\ given in \cite[Theorem 4]{rd}, but we presently have no general (say local or semi-local) criterion ensuring exponential mesoscopic rank (compare Section \ref{more} below). Another interesting problem is to prove or disprove the existence of groups of mesoscopic rank for which the mesoscopic rank profile  at some vertex grows faster than polynomials but slower than exponential functions.
\end{remark}

\section{Proof of Theorem \ref{meso}}\label{s3'}

We prove that the Braid group $B_4$ and its central quotient $G=B_4/\langle z \rangle$ are of exponential mesoscopic rank, respectively, in Section \ref{63} and Section \ref{62}. 

\subsection{A closer look at the 4-string complexes $Y$ and $X$.} Let us first recall in some more details the description of the Brady action of $B_4$ on $Y$ and  its quotient action of $G$ on $X$, following \cite{Brady} and \cite{crisp}. 
Consider the following presentation of $B_4$,
\begin{align*}
B_4=\langle a,b,c,d,e,f\mid &ba=ae=eb,\, de=ec=cd,\\
&bc=cf=fb,\, df=fa=ad,\\
&ca=ac,\, ef=fe\rangle,
\end{align*}
and let us keep the notations  $x=bac$ and $y=bac^2$, so that $x^4=y^3=z$ generates the center of $B_4$. There are exactly sixteen ways to write $x$ as a product of three of the generators $a,\ldots, f$. These can be expressed as the length 3 subwords of the following two words of length 12: 
\[
W_1=bcadefbacdfe;\hskip1.5cm W_2=faecfaecfaec,
\]
which are representative for the central element $x^4=y^3=z$ in $B_4$ (see \cite[page 139]{crisp}). To each of these expressions $x=a_1a_2a_3$ one associates an Euclidean tetrahedron whose faces are right-angled triangles, and whose edges have length 
\[
|x|=\sqrt3;~~ |a_i|=1;~~  |a_1a_2|=|a_2a_3|=\sqrt2.
\]
The corresponding labelled tetrahedra can be assembled to form a compact complex $V$ such that $\pi_1(V)=B_4$. Then $Y=\tilde V$ is the universal cover of $V$ with the corresponding deck-transformation action of $B_4$.

The CAT(0) space $Y$ splits as a metric product $Y=X\times \RI$, where $X$ is the range of a projection of $Y$ along the $z$-axis. 
The image of each tetrahedron in $Y$ under this projection is an Euclidean equilateral triangle in $X$ and the action of $B_4$ factors out to a simplicial action of $G=B_4/\langle z\rangle$ on $X$, which is proper and cocompact.   The Cayley graph of $G$ with respect to the generating set $S=\{a,\ldots, f\}$ (where the above $a,\ldots, f$ are viewed as elements of $G$ under a slight abuse of notation)  is a 4-to-1 cover of the 1-skeleton of $X$. 


Links at vertices in  $X$ are represented on Figure \ref{fig1} below, where the right hand side representation corresponds to Figure 3 in \cite{Brady} and  Figure 6 in \cite{crisp}. The equivalent left hand side representation is included  for future reference (see Section \ref{more}).

\begin{figure}[htbp]
\centerline{\includegraphics[width=13cm]{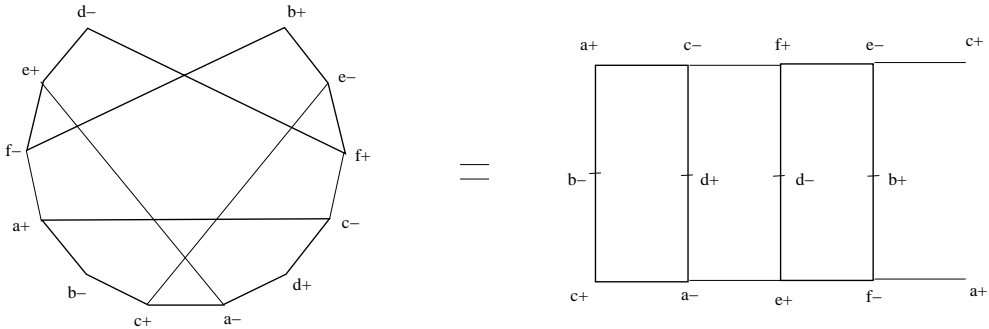}}
\caption{The link $L$ and its labelling}\label{fig1}
\end{figure}
The labellings on these figure corresponds to edges of the generating set $S$ entering and leaving the given vertex (this depends on the choice of a representative of the coset of $\langle x\rangle$ in $G$, but a different choice will simply relabel the link according to the action of an element of stabilizer of the vertex, see \cite[p. 141]{crisp}).

We call \emph{lozenge} in $X$ the reunion of two triangles glued along an edge of valence 2, and restrict from now on the terminology \emph{triangle of $X$} to  those equilateral triangles in $X$ which are not included in a lozenge.  According to the description given on p. 160 of \cite{crisp}, the complex $X$ is built out of triangles and lozenges, all of whose edges being trivalent in $X$ and labelled in the same way. There are three types of corners: in triangles, all angles labelled by 1, while in lozenges the angles are labelled 2 or 3 depending of whether it equals $\pi/3$ or $2\pi/3$. Then triangles and lozenges in the complex  $X$ are arranged in such a way that the labelled link at each vertex matches that given on the following Fig. \ref{figX} (our notations differ slightly from those of \cite{crisp}, in particular our label 3 correspond to $2T_3^-$ and $2T_3^+$ in \cite[Fig. 19]{crisp}). 

\begin{figure}[htbp]
\centerline{\includegraphics[width=13.5cm]{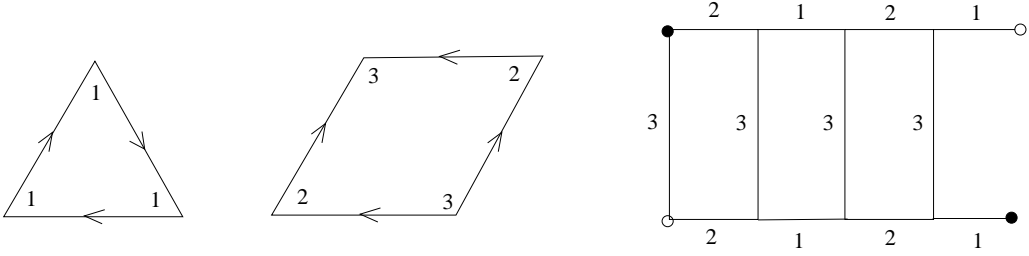}}
\caption{Description of the complex $X$}\label{figX}
\end{figure}

\subsection{Exponential mesoscopic rank for $X$ and   $G=B_4/\langle z\rangle$}\label{62} The proof will follow the strategy of \cite{rd} (see Section 6.1 and 6.2), with the additional difficulty that the link $L$  embeds into the incidence graph of the Fano plane (compare  Section \ref{more}). 
Let us first derive a few elementary lemmas regarding the local structure of $X$.

\begin{lemma}\label{loz}
Let $R$ be a lozenge of $X$. Any boundary edge  of $R$ is incident to exactly a triangle and a lozenge $R'\neq R$ of $X$. Furthermore, $R\cup R'$ is isometric to a parallelogram and we will say that $R$ and $R'$ are \emph{aligned} in $X$.
\end{lemma}

\begin{proof}
Let $R$ be a lozenge of $X$, $e=[A,B]$ be a boundary edge of $R$, where $A$ is the vertex of $e$ whose internal angle in $R$ is $2\pi/3$. The geometry of $L$ shows that there are two faces incident to $e$ which are not included in $R$. Inspecting the  link at $B$, we see that one of these faces is a triangle of $X$, while the other is a lozenge whose internal angle at $B$ is $2\pi/3$. This follows from the fact that every vertex of valence of 3 in $L$ is adjacent to a vertex of valence 2, and the vertices with valence 2  are at distance $\geq 3$ one from the other. Hence the lemma is proved.
\end{proof}

\begin{lemma}\label{hexa}
Let $R$ be a lozenge of $X$ and $A$ be a vertex of $R$ of internal angle $2\pi/3$. There are exactly two lozenges $R_1$ and $R_2$ in $X$ such that $R\cap R_1=R\cap R_2=\{A\}$ and such that both $R\cup R_1$ and $R\cup R_2$ are included in a flat hexagon of $X$. 
(This hexagon contains $R$ and $R_1$ (resp. $R$ and $R_2$) and the two triangles of $X$ containing $A$ and completing $R$ and $R_1$ (resp. $R_2$) to a local flat at $A$.)
\end{lemma}

\begin{proof}
The assertion follows from the fact that, given a vertex $x$ of valence 2 in $L$, there are exactly two vertices $y$ and $z$ of valence 2 in $L$ such that $(x,y)$ on the one hand, and $(x,z)$ on the other, are at distance $\pi$ in a cycle of $L$ of length $2\pi$.   
\end{proof}

\begin{lemma}\label{Lpi}
Let $x$ and $y$ be two trivalent vertex at distance $\pi$ in  $L$. Then there are precisely three distinct paths of length $\pi$  with extremities $x$ and $y$ in $L$. Depending on the position of $x$ and $y$ in $L$ these paths   are  labelled in either one of the following two ways (with the labelling given by Fig. \ref{figX}):

\begin{itemize} 
\item Case I:  2-3, ~ ~3-2, ~ ~and~ ~ 1-2-1;
\item Case II: 1-3, 3-1, and 2-1-2.
\end{itemize}
\end{lemma}

\begin{proof}
It is easily seen from the geometry of the link (Fig. \ref{figX}) that trivalent points at distance $\pi$ in $L$ can be joined by simplicial paths whose edges labelling are:
\begin{itemize}
\item[(a)] 1-3 (or 3-1)
\item[(b)] 2-3 (or 3-2)
\item[(c)] 1-2-1
\item[(d)] 2-1-2
\end{itemize}
and furthermore that there are precisely three distinct paths between any two such points. Indeed, the group of labelling preserving automorphisms of  $L$  is homogeneous on trivalent vertices, so we can assume for instance that $x=a^+$ (in the notation of Fig. \ref{fig1}), in which case there are two possibilities for $y$, namely $y=a^-$ and $y=e^-$. The lemma follows, where case I corresponds to $y=a^-$ and case II to $y=e^-$.
\end{proof}

We call \emph{singular geodesic} in $X$ a CAT(0) geodesic of $X$ which is included in the 1-skeleton of $X$ (viewed with respect to the triangle/lozenge simplicialization, in particular, all edges of singular geodesics are of valence 3). 
It is easy to see that for $u=a$ or $u=c$, and every vertex $A$ in $X$, the vertices $u^iA$, $i\in \ZI$,  belong to a  singular geodesic of $X$. 
   Indeed, since the labellings of the link at each vertex $u^iA$ are given by permuting letters of $L$, it is sufficient to show that  the points of $L$ with label $u^-$ and $u^+$ are trivalent vertices at distance $\pi$ in $L$, which straightforward. We will denote this geodesic by $u^\ZI A$.

Recall that a  subset $S$ of $X$ is called a (flat) \emph{strip}  if it is isometric to a product $I\times \RI\subset \RI^2$ where $I$ is a compact interval of $\RI$. The boundary of $S$ is a reunion of two parallel geodesics of $X$, say $d$ and $d'$, and is denoted $(d,d')$. The \emph{height} of $S$ is the CAT(0) distance between $d$ and $d'$.

The following lemma asserts that singular geodesic in $X$ all appear as branching locus of flat strips of $X$. This property  is reminiscent of  affine Bruhat-Tits buildings (say, of dimension 2), where it is true in a somewhat stronger form (in particular strips may be extended arbitrarily there).

\begin{lemma}\label{smallstrip} Let $d$ be a singular geodesic of $X$. There are precisely three flats strips in $X$ of height at least $\sqrt 3/2$ whose pairwise intersection are reduced to $d$. 
\end{lemma}

\begin{proof}
For each edge $e$ of $d$ consider the three faces $T_e^i$, $i=1,2, 3$, whose boundary contains $e$. Two of these faces are lozenges and one of them, say $T_e^1$, is a triangle of $X$ (see Lemma \ref{loz}). Let $f$ be an edge of $d$ adjacent to $e$ and let $A$ be their intersection point. The points corresponding to $e$ and $f$ in the link $L_A$ of $X$ at $A$ are trivalent, and it can be easily checked that they are at distance $\pi$ from each other in $L_A$. Thus Lemma \ref{Lpi} applies. In case I, the faces $T_e^1$ and $T_f^1$ correspond to a path of length $\pi$ of the form 1-2-1, and, up to permutation of indices $i$, the faces $T_e^i$ and $T_f^i$ (for $i=2,3$) corresponds to a cycle of length $\pi$ of the form 2-3 and 3-2. In case 2 and again up to permutation of the indices $i$, the faces $T_e^1$ and $T_f^2$ correspond to a path of length $\pi$ of the form 1-3, the faces $T_f^1$ and $T_e^2$ correspond to a path of length $\pi$ of the form 3-1, while the faces $T_e^3$ and $T_f^3$ correspond to a path of length $\pi$ of the form 2-1-2. Then the lemma follows by iterating this on both sides of the geodesic $d$ starting from a fixed edges $e$. The height of each strip may be taken to be at least $\sqrt 3/2$.
\end{proof}

We say that a vertex of a singular geodesic of type I (resp. of type II) depending on whether case 1 (resp. case 2) applied in the proof of the above lemma, and call a geodesic  of type I (resp. of type II) if all its vertices are of type I (resp. type II). For instance the geodesic $a^\ZI A$ and $c^\ZI A$ are of type I for any vertex $A$ of $X$.

\begin{lemma}\label{strip} Let $d$ be a singular geodesic of type I in $X$. There are precisely three flats strips in $X$ of minimal height whose pairwise intersection are reduced to $d$ and whose boundary geodesics are singular geodesics type I in  $X$. Two of them have height $\sqrt 3/2$, and are reunions of aligned lozenges (see Lemma \ref{loz}), and the other one has height $\sqrt 3$, and is a reunion of hexagons as described in Lemma \ref{hexa}, and triangles of $X$ which are the unique triangles completing these hexagons to a flat strip. 
\end{lemma}

\begin{proof}
Let $d$ be a singular geodesic of type I in $X$ and let $S_1$, $S_2$ and $S_3$ be the strips of height $\sqrt 3/2$ given by Lemma \ref{smallstrip}, whose pairwise intersections are reduced to $d$. We may assume at each vertex $A$ of $d$ the path of length $\pi$ in $L_A$ corresponding to  $S_1$  are of the form 1-2-1. Then the path corresponding to $S_2$ and $S_3$ are either of the form 2-3 or 3-2.  

\begin{figure}[htbp]
\centerline{\includegraphics[width=13cm]{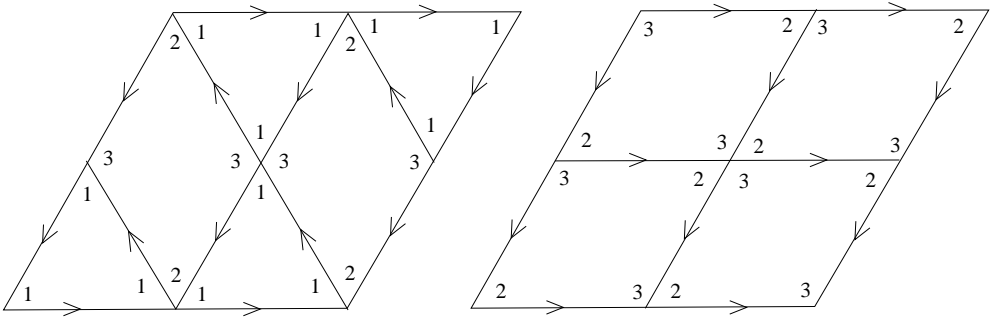}}
\caption{Parallelograms and strips on type I geodesics of $X$}\label{figpar}
\end{figure}

Let us first consider the strip $S_1$ and for each vertex $A$ of $d$ denote by $R_A$ the lozenge of $X$ corresponding to the index 2 in the path 1-2-1 of $L_A$. It is easily seen that if $A$ and $B$ are consecutive vertices on $d$, then the lozenges $R_A$ and $R_B$ are in the configuration described in Lemma \ref{hexa} and they can be completed by a unique triangle of $X$ (besides the one in $S_1$) to form an hexagon $H_{AB}$. The reunion $S_1'$ of all hexagon $H_{AB}$ when $A$ and $B$ runs over the pair of adjacent vertices on $d$ is a flat strip of $X$ of height $\sqrt 3$. Furthermore it is a simple matter to check (with Lemma \ref{Lpi}) that all the vertices of the boundary of this strip which is distinct from $d$ are of type I, which proves the assertion of the Lemma in that case. A parallelogram of the strip $S_1'$ is represented on Fig. \ref{figpar} on the left.

Consider now the strip $S_2$, which is of height $\sqrt 3/2$. The boundary of this strip which is distinct from $d$ contains only vertices whose link intersect $S_2$ along a path of the form 2-3 or 3-2. By Lemma \ref{Lpi} again, these vertices are of type I.
 The case of $S_3$ being identical to that of $S_2$, this proves the lemma. Parallelograms of the strips $S_2$ and $S_3$ are represented on Fig. \ref{figpar}.
\end{proof}

\begin{theorem}\label{mesoGth}
The complex $X$ is of exponential mesoscopic rank at every vertex. More precisely let $O$ be a vertex of $X$ and $k$ be a sufficiently large integer (in fact $k\geq 32$ is sufficient for our purpose). Then the mesoscopic profile $\varphi_O$ at $X$ satisfies
\[
\varphi_O\geq \left ({3\over 2}\right)^{2\mu_k-4}
\]
on the interval $[k-1, k]$ of $\RI_+$, where
\[
\mu_k= \left \lceil k({2\over \sqrt3 }-1) + ({2\over \sqrt 3}-3)\right \rceil.
\]
In particular the group $G$ is of exponential mesoscopic rank.
\end{theorem}

\begin{proof}
Let $\Pi$ be the flat containing the origin $O=O_0$ of $X$ and generated by the subgroup $\langle a,c\rangle\simeq \ZI^2$ of $G$. Denote  $O_1=(ac)^{-1}(O_0)$ let $d=[O_0,\infty)$ be the semi-line of $\Pi$ of origin $O_0$ and containing $O_1$. Hence the vertices of $d$ are the points 
\[
O_k=(ac)^{-k}(O_0)
\]
for $k\in \NI$. Let $\Pi_0$ be the sector of $X$ of extremity $O_0$, of angle $2\pi/3$ at $O_0$, and which is bisected by the semi-line $d$ (see Figure \ref{mesoG}). The boundary of $\Pi_0$ is included in the reunion of singular geodesics $d_1$ and $d_2$ which intersects at $O_0$; the first one  contains the vertices $a^{-k}(O_0)$ and the second one the vertices $c^{-k}(O_0)$, $k\in \NI$. Both $d_1$ and $d_2$ are of type I.

Consider the vertices  $A=a(O_0)$ of the flat $\Pi$.
By Lemma \ref{strip}:
\begin{enumerate}
\item There is a unique strip $S_1$ of height $\sqrt 3$ whose intersection with $\Pi$ is reduced to $a^{\ZI}(O_0)$, and whose other boundary $d_1'$ is a singular geodesic of type I in $X$. We consider then on $d_1'$ the unique strip of height 1, say $S_1'$, which corresponds in the link of vertices of $d_1'$ to paths of the form 3-2 (see Fig. \ref{mesoG}). Let $d_1''$ be the other boundary of $S_1'$.
\item Consider the  strip $S_2$ in $X$ of height $\sqrt 3/2$  on $c^{\ZI}(O_0)$ which contains $A$. (This strip is included in $\Pi$ and its other boundary  $d_2'=c^{\ZI}(A)$ is a singular geodesic of type I in $X$.) There is on  $d_2'$ a unique strip $S_2'$ of height $\sqrt 3$ whose other boundary $d_2''$ is a singular geodesic of type I in $X$.
\end{enumerate}

The strips $S_1$, $S_1'$, $S_2$ and $S_2'$ are represented on Fig. \ref{mesoG}, together with the labellings given by the links at their vertices.

\begin{figure}[htbp]
\centerline{\includegraphics[width=14cm]{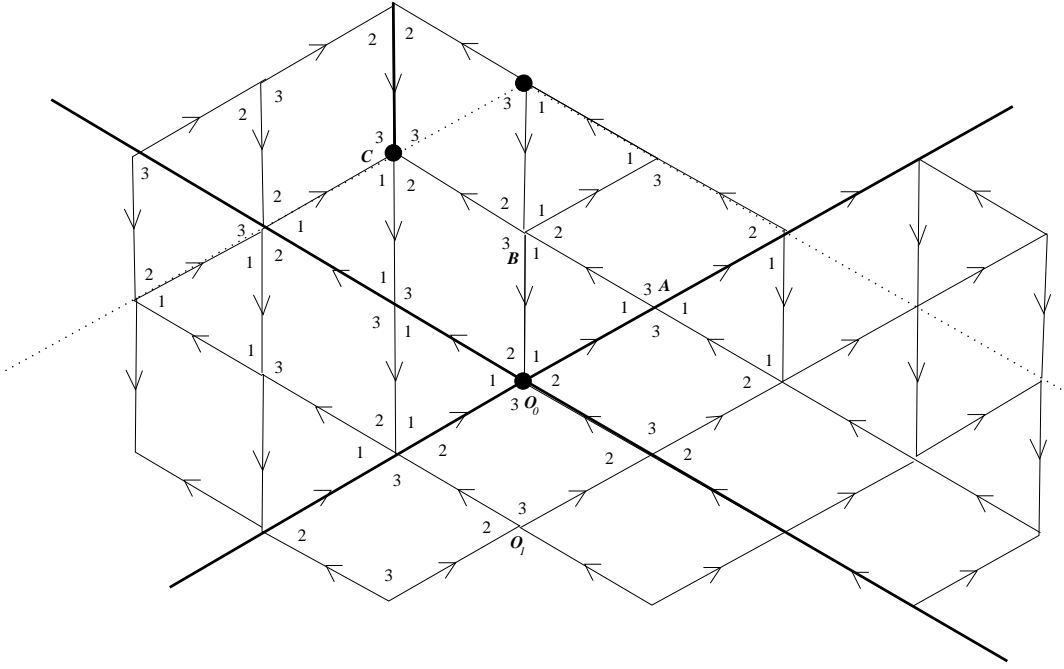}}
\caption{Exponential mesoscopic rank of the complex $X$}\label{mesoG}
\end{figure}

\begin{lemma}\label{c13}
Let $k\in \NI$ and let $D$ be a flat disk in $X$ of center $O_k$ such that $D\backslash (X\backslash \Pi_0) =D\cap \Pi_0$. If the intersections $D\cap S_i$ and $D\cap S_i'$, $i=1,2$ are non empty open sets, then $D$ is not included in a flat of $X$. 
\end{lemma}

\begin{proof}[Proof of Lemma \ref{c13}]
As we see on the link $L_A$ of $A$, there is a unique lozenge $R$, which corresponds to a label 2 in $L_A$ and which  extends the strips $S_1$ and $S_2'$ at the point $A$ to a flat disk  in $X$ containing $A$ as an interior point. This lozenge contains a vertex $B$ at distance $\pi$ from $c^{-1}(A)$ in $L_A$ and in turn, there is a unique way to extend  the resulting configuration to a flat disk in $X$ containing $B$ as an interior point. This disk corresponds to a circle of length $2\pi$ which is labelled 1-3-2-1-2 in $L_B$.  Let $R'$ be the lozenge of $X$ distinct from $R$ which corresponds the label 2 in this circle.

It is easy to see that, if $D$ is a flat disk as in the statement of the lemma, then any flat disk $D'$ of center $O_k$ and radius $> k+1$ which contains $D$ must contain the points $A$ and $B$ as interior points and must intersect the lozenge $R'$ along a non-empty open subset. On the other hand $D$, and a fortiori $D'$, intersects the strip $S_1'$ along a non empty open set. Thus $D'$ intersects along an non empty open set the lozenge of $S_1'$ which contains $C$ and whose internal angle at $C$ is $2\pi/3$.  But this shows that $D'$ cannot be extended beyond the point $C$, since this would give a cycle of length $2\pi$ in the link $L_C$ containing two successive edges of length $2\pi/3$. Thus neither $D'$ nor $D$ is included in a flat of $X$.  
\end{proof}

We can now conclude the proof of Theorem \ref{mesoGth}. We proceed as in Lemma 59 of \cite{rd}, to which we refer for more details. For $k\geq 32$ let $\mu_k$ be the integer defined in the statement of the theorem and let $\nu_k=\left (3\over 2\right )^{\mu_k}$. (Since $k\geq 32$ we have $\mu_k\geq 3$.)
Using Lemma \ref{strip}, we can construct, for $i=1,2$, (at least) $\nu_k$  distinct flat strips 
\[
\S_i^1,\ldots, S_i^{\nu_k}
\]
in $X$ of height ${\sqrt 3\over 2}\mu_k$, each of whose intersection with $S_i'$ is reduced to $d_i''$.
 (The lower bound $\nu_k$  is estimated by examining transverse trees in the sets $\cup_{j=1}^{\nu_k} S_i^j$;  sharper bounds can  be obtained  easily but $\nu_k$ is enough to show exponential growth of the mesoscopic profile.)   So let $i=(i_1,i_2)\in \{1,\ldots,\nu_k\}^2$ and consider the subset $\Pi_i$ of $X$ defined by \[
 \Pi_i=\Pi_0\cup S_1^{i_1}\cup S_2^{i_2}.
 \]
Then  the set $D_i$ of points of $\Pi_i$ at distance  $\leq k+1$ from $O_k$ in $\Pi_i$ is a flat disk in $X$ whose boundary contains $B$. Furthermore   the disks $D_i$ are pairwise distinct when $i$ varies in  $\{1,\ldots,\nu_k\}^2$. 
For $r\in [0,k+1]$ write $D_i^r$ for the concentric disk of radius $r$ in $D_i$.  Then for any fixed $r\in [k,k+1]$ the family of disks 
\[
\{D_1^r, \ldots D_{\nu_k}^r\}
\]
contains at least $\left ({3\over 2}\right) ^{2\mu_k -4}$
distinct elements. Furthermore all these disks satisfy the assumption of Lemma \ref{c13} and thus are not included in a flat of $X$. Since the vertex $O_k$ are all equivalent under the group $G$, this proves the theorem.
\end{proof}

\subsection{Exponential mesoscopic rank for $Y$ and the braid group $B_4$.}\label{63} We conclude this section with the proof of Theorem \ref{meso}.

\begin{theorem}
Let $O$ be a vertex of $Y$ and consider the CAT(0) projection $\pi : Y\to X$ associated to the metric decomposition $Y\simeq X\times \RI$.  Then the mesoscopic profile $\varphi_O^Y$ at $Y$ satisfies
\[
\varphi_O^Y\geq \varphi_{\pi(O)}^X
\]
where $\varphi_{O'}^X$ is the mesoscopic profile of $X$ at a vertex $O'\in X$. In particular the braid group $B_4$ is of exponential mesoscopic rank.
\end{theorem}

\begin{proof}
For $r\in \RI$, let $k=\varphi_X(r)$ and consider $k$ distinct flat disks $D_1,\ldots, D_k$ of center $O'=\pi(O)$ and radius $r$ in $X$ which are not included in a flat of $X$. Let 
\[
C_i=\pi^{-1}(D_i)\simeq D_i\times \RI
\]
be the cylinder of $Y$ corresponding to the decomposition $Y\simeq X\times \RI$. These cylinder are isometric to cylinders in the Euclidean space $\RI^3$ and in particular the ball of center $O$ and radius $r$ in $C_k$ are all flat balls of $X$. Furthermore, these balls are not included in flats of $Y$. Indeed, if $B_i\subset \Pi$ where $\Pi\simeq \RI^3$ is isometric to the Euclidean space $\RI^3$, then the projection $\pi(\Pi)$ is a convex subset of $X$ which is isometric to the Euclidean space $\RI^2$. But this shows that $D_i$ is included in a flat of $X$. Finally, as the ball $B_i$ are pairwise distinct (since the disks $D_i$ are), we obtain that $\varphi_Y(r)\geq k$ as claimed. The last assertion follows from Theorem \ref{mesoGth}.
\end{proof}

\begin{remark}
It would be interesting to give example of groups which act properly with compact quotient on a CAT(0) space of dimension $\geq 3$ of exponential mesoscopic rank, which doesn't split as a metric product where some factor is of exponential mesoscopic rank. 
 In view of Tits' classification of affine buildings, it seems plausible that CAT(0) simplicial complexes ``whose rank is close to their dimension"   will get sparse when the dimension gets strictly greater than 2. Recall here that affine Bruhat-Tits buildings are completely classified in dimension $\geq 3$ by work of Tits \cite{Tits74}, and that this is far from being possible in dimension 2 which offers a great degree of freedom \cite{henri}.
 We also refer to the paper of Ballmann and Brin \cite{BB} concerning rank rigidity results in dimension 3.
\end{remark}

\section{More on mesoscopic rank}\label{more}

In the present section we investigate possible relations between the Brady complex $X$ and triangle buildings of order 2. Furthermore we present a group   which acts freely isometrically with compact quotient on polyhedra of exponential mesoscopic rank that is embeddable into a triangle building.

Let us observe first that 
the link of $X$ (represented Fig. \ref{fig1}) obviously embeds (simplicially) into the incidence graph $L_2$ of the Fano plane. Such an embedding is made explicit on Fig. \ref{fano} below; the graph $L$ is obtained from $L_2$ by removing a tree $T$ of 5 edges (indicated in dots on Fig. \ref{fano}). We call \emph{center-edge} of $T$ the only non-extremal edge of $T$---this tree $T$ appears often in \cite{poisson} where  it is called a fishy edge (or a fish bone, depending on the translation).

\begin{figure}[htbp]
\centerline{\includegraphics[width=5.5cm]{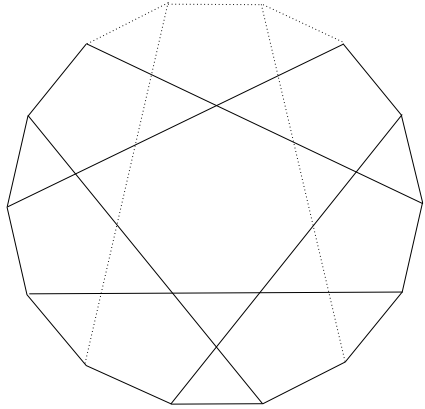}}
\caption{The incidence graph $L_2$ of the Fano plane and the link of $X$}\label{fano}
\end{figure}

The graph $L_2$ is a spherical building and can be identified to the link of triangle buildings  of order 2 (see e.g. \cite{Ronan}; triangle buildings are also called affine Bruhat-Tits buildings of type $\tilde A_2$). By \cite{henri}  there are uncountably many such buildings, and their  groups  of automorphisms is generically trivial (generic is taken here in the sense of Baire with respect to some appropriate topology).  In view of the above embedding $L\inj L_2$, it is natural to ask whether the complex $X$ itself can be simplicially embedded into one of these triangle buildings. It turns out that this problem has an elementary answer.

\begin{proposition}\label{prop20}
Let $X$ be the brady complex and $\Delta$ be a triangle building of order 2. There is no simplicial embedding $X\inj \Delta$. More generally, any CAT(0) complex $X$ of dimension 2 whose faces are equilateral triangle and whose links at each vertex are isomorphic to $L$ does not embed simplicially into a triangle building. 
\end{proposition}

\begin{proof}
Let $X$ be a  CAT(0) complex $X$ of dimension 2 whose faces are equilateral triangle and whose links at each vertex are isomorphic to $L$, and assume that we are given  a simplicial embedding $X\inj \Delta$. For a vertex $D$ of $X$ we write $T_D$ for the removed tree, 
\[
T_D=L_{\Delta,D}\backslash L_{X,D},
\]
where $L_{Z,D}$ denotes the link of $D$ in the complex $Z$. Fix some vertex $A\in X$. Then there is a unique triangle  in $\Delta$, say $(ABC)$, which corresponds to the center-edge of the tree $T_A$ at the point $A$. Denote by $(ABB')$ and $(ABB'')$ the two other triangles in $\Delta$ adjacent to the edge \begin{figure}[htbp]
\centerline{\includegraphics[width=6cm]{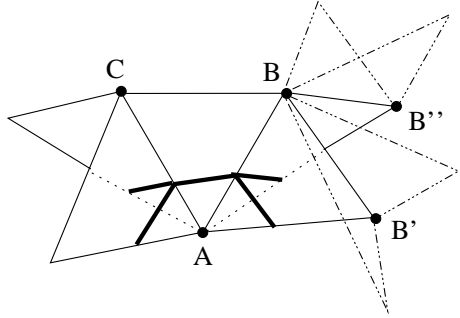}}
\caption{Embedding $X$ into a triangle building}\label{fig6}
\end{figure}
$[A,B]$. In the link $L_{\Delta,B'}$ (resp. $L_{\Delta,B''}$),  the edge corresponding to the triangle $(ABB')$ (resp. $(ABB'')$) is extremal in the tree $T_{B'}$ (otherwise $L_{\Delta,A}\backslash L_{X,A}$ would contain more than 5 edges). Thus the three triangles of $\Delta$ adjacent to the edge $[B,B']$ (resp. $[B,B'']$) do not belong to $X$. But then the graph $L_{\Delta,B}\backslash L_{X,B}$ contains at least six edges, contradicting our assumptions. Therefore  there is no simplicial embedding   $X\inj \Delta$.
\end{proof}

\begin{remark}
The proof of Proposition \ref{prop20} shows more, namely, it shows that the obstruction of an embedding $X\inj \Delta$ is \emph{local}: for $X$ and $\Delta$ as in the proposition, there is no simplicial embedding of simplicial balls of radius 2 in $X$ into simplicial balls of radius 2 in $\Delta$. In other words the embedding $L\inj L_2$ is the best  we can do.
\end{remark}

We will conclude with a discussion of the following question, which is natural in view of the  above. 

\medskip

Does (exponential) mesoscopic rank for a CAT(0) complex prevents  simplicial embeddings of this complex 
into an affine Bruhat-Tits building ? 

\medskip

It turns out that the answer   is negative. It can be shown that the group $G_0$ defined by the presentation: 
\[
G_0=\langle r,s\mid s^{-2}ts^2t=t^2st^{-2}\rangle
\]
admits a free and isometric action  with compact quotient on a CAT(0) simplicial complex $X_0$ of dimension 2 such that:
\begin{enumerate}
\item there is a simplicial embedding $X_0\inj \Delta$ where $\Delta$ is a triangle building of order 2;
\item $X_0$ is of exponential mesoscopic rank.
\end{enumerate}

Since the construction from which $G_0$ is derived is not related  to braid groups and would take us too far away from the subject of the present paper,  we will omit the proofs of the above two statements. 
Let us simply describe the local geometry of $X_0$, which can be interestingly  compared to that of the Brady complex $X$.

\begin{figure}[htbp]
\centerline{\includegraphics[width=5.5cm]{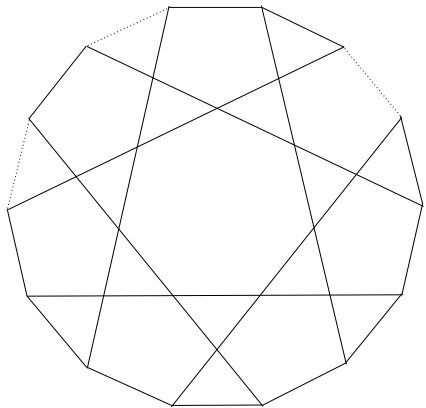}}
\caption{The link of a complex of mesoscopic rank that can be embedded into a triangle building}\label{fanomeso}
\end{figure}

The links at each vertex of $X_0$ (which necessarily embed  simplicially into $L_2$) are all isometric (in fact $G_0$ acts transitively on the vertices of $X_0$).  They are obtain from graph $L_2$ by removing 3 edges. Note that these edges are not in a tree but they are irregularly positioned on the graph $L_2$, as for the link $L$ of the Brady complex $X$. The complex $X_0$ has exponential asymptotic rank in the sense of \cite{rd}.

The above figure represents the links in $X_0$.
The complex $X_0$ itself contains two types of faces, equilateral triangles and parallelograms of size $2\times1$ in the Euclidean plane;  a representation of these faces and their labellings can be found on Fig. \ref{x0}.

\begin{figure}[htbp]
\centerline{\includegraphics[width=11cm]{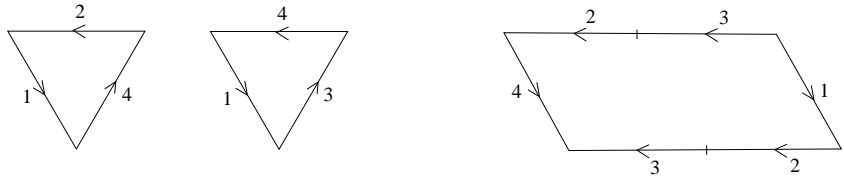}}
\caption{Description of the complex $X_0$}\label{x0}
\end{figure}

\end{document}